\documentclass[twoside,leqno,10pt, A4]{amsart}
\usepackage{amsfonts}
\usepackage{amsmath}
\usepackage{amscd}
\usepackage{amssymb}
\usepackage{amsthm}
\usepackage{amsrefs}
\usepackage{latexsym}
\usepackage{mathrsfs}
\usepackage{bbm}
\usepackage{color}
\begin{document}

\newtheorem{theorem}[subsection]{Theorem}
\newtheorem{proposition}[subsection]{Proposition}
\newtheorem{lemma}[subsection]{Lemma}
\newtheorem{corollary}[subsection]{Corollary}
\newtheorem{conjecture}[subsection]{Conjecture}
\newtheorem{prop}[subsection]{Proposition}
\numberwithin{equation}{section}
\newcommand{\mr}{\ensuremath{\mathbb R}}
\newcommand{\dif}{\mathrm{d}}
\newcommand{\intz}{\mathbb{Z}}
\newcommand{\ratq}{\mathbb{Q}}
\newcommand{\natn}{\mathbb{N}}
\newcommand{\comc}{\mathbb{C}}
\newcommand{\rear}{\mathbb{R}}
\newcommand{\prip}{\mathbb{P}}
\newcommand{\uph}{\mathbb{H}}
\newcommand{\fief}{\mathbb{F}}
\newcommand{\majorarc}{\mathfrak{M}}
\newcommand{\minorarc}{\mathfrak{m}}
\newcommand{\sings}{\mathfrak{S}}
\newcommand{\fA}{\ensuremath{\mathfrak A}}
\newcommand{\mn}{\ensuremath{\mathbb N}}
\newcommand{\mq}{\ensuremath{\mathbb Q}}
\newcommand{\half}{\tfrac{1}{2}}
\newcommand{\f}{f\times \chi}
\newcommand{\summ}{\mathop{{\sum}^{\star}}}
\newcommand{\chiq}{\chi \bmod q}
\newcommand{\chidb}{\chi \bmod db}
\newcommand{\chid}{\chi \bmod d}
\newcommand{\sym}{\text{sym}^2}
\newcommand{\hhalf}{\tfrac{1}{2}}
\newcommand{\sumstar}{\sideset{}{^*}\sum}
\newcommand{\sumprime}{\sideset{}{'}\sum}
\newcommand{\sumprimeprime}{\sideset{}{''}\sum}
\newcommand{\shortmod}{\ensuremath{\negthickspace \negthickspace \negthickspace \pmod}}
\newcommand{\V}{V\left(\frac{nm}{q^2}\right)}
\newcommand{\sumi}{\mathop{{\sum}^{\dagger}}}
\newcommand{\mz}{\ensuremath{\mathbb Z}}
\newcommand{\leg}[2]{\left(\frac{#1}{#2}\right)}
\newcommand{\muK}{\mu_{\omega}}

\title[First Moment of {H}ecke {$L$}-functions with quartic characters]{First Moment of {H}ecke {$L$}-functions with quartic characters at the central point}

\date{\today}
\author{Peng Gao and Liangyi Zhao}

\begin{abstract}
In this paper, we study the first moment of central values of Hecke $L$-functions associated with
quartic characters.
\end{abstract}

\maketitle

\noindent {\bf Mathematics Subject Classification (2010)}: 11M41, 11L40  \newline

\noindent {\bf Keywords}: quartic Hecke characters, Hecke $L$-functions

\section{Introduction}

Because of the arithmetic information that they carry, the values of $L$-function at $s=1/2$ are very important in number theory.  The average of these values over a family of characters of a fixed order has been studied extensively.  For example, M. Jutila \cite{Jutila} gave the evaluation the mean value of $L(1/2, \chi)$ for quadratic Dirichlet characters.  The error term in the asymptotic formula in \cite{Jutila} was later improved by \cites{DoHo, MPY, ViTa}.  In \cite{B&Y}, S. Baier and M. P. Young studied the moments of $L(1/2, \chi)$ for cubic Dirichlet characters.  Literature also abounds in the investigation of moments of Hecke $L$-functions associated with various families of characters of a fixed order \cites{FaHL, GHP, FHL, Luo, Diac}.   In this paper, we shall investigate the first moment of $L$-functions associated with a family of quartic Hecke characters.  \newline

Let $K =\mq(i)$.  It is well known that $K$ has class number $1$, and in the ring of
integers $\mathcal{O}_K = \mz[i]$ every ideal coprime to $2$ has a unique
generator congruent to $1 \pmod {(1+i)^3}$. For $1 \neq c \in \mathcal{O}_K$ which
is square-free and congruent to $1 \pmod {16}$, let $(\frac {\cdot}{c})_4$ be the quartic
residue symbol. Since $\chi_c =(\frac {\cdot}{c})_4$ is trivial on
units, it can be regarded as a primitive character of the ray
class group $h_{(c)}$. We recall here that for any $c$, the ray
class group $h_{(c)}$ is defined to be $I_{(c)}/P_{(c)}$, where
$I_{(c)} = \{ \mathcal{A} \in I : (\mathcal{A}, (c)) = 1 \}$ and
$P_{(c)} = \{(a) \in P : a \equiv 1 \pmod{c} \}$ with $I$ and $P$
denoting the group of fractional ideals in $K$ and the subgroup of
principal ideals, respectively. The Hecke $L$-function associated
with $\chi_c$ is defined for $\Re(s) > 1$ by
\begin{equation*}
  L(s, \chi_c) = \sum_{0 \neq \mathcal{A} \subset
  \mathcal{O}_K}\chi_c(\mathcal{A})(N(\mathcal{A}))^{-s},
\end{equation*}
  where $\mathcal{A}$ runs over all non-zero integral ideals in $K$ and $N(\mathcal{A})$ is the
norm of $\mathcal{A}$. As shown by E. Hecke, $L(s, \chi_c)$ admits
analytic continuation to an entire function and satisfies the
functional equation (\cite[Theorem 3.8]{HIEK})
\begin{equation}
\label{1.1}
  \Lambda(s, \chi_c) = W(\chi_c)(N(c))^{-1/2}\Lambda(1-s, \overline{\chi}_c),
\end{equation}
   where
\begin{equation}
\label{1.2}
   W(\chi_c)=\sum_{a \in \mathcal{O}_{K}/(c)}\chi_c(a)e\Big ( \text{Tr}\Big (\frac {a}{\delta c}\Big )\Big ),
\end{equation}
   with $(\delta)=(\sqrt{-4})$ being the different of $K$, and
\begin{equation*}
  \Lambda(s, \chi_c) = (|D_K|N(c))^{s/2}(2\pi)^{-s}\Gamma(s)L(s, \chi_c),
\end{equation*}
   where $D_K=-4$ is the discriminant of $K$. We also note that
   $W(\chi_c)=g(c)$, where $g(c)$ is the Gauss sum, defined later in
   Section \ref{sec2.4}. The above discussions also apply to $c=1$, provided that we interpret $\chi_1$ as the principal character $\pmod 1$ so that $\Lambda(s, \chi_1)$ becomes $\zeta_{\mq(i)}(s)$, the Dedekind zeta function of $\mq(i)$, a convention we shall follow throughout the paper. \newline

Our goal here is to establish the following theorem.
\begin{theorem}
\label{firstmoment}
  For $y \rightarrow \infty$ and any $\varepsilon > 0$, we have
\begin{equation} \label{1stmom}
   \sumstar_{c \equiv 1 \pmod {16}}L\left( \frac{1}{2},
   \chi_c \right)\exp \left( - \frac{N(c)}{y} \right)=Ay+O_{\varepsilon}(y^{13/14+\varepsilon}) .
\end{equation}
   Here $A$ is an explicit constant given in \eqref{Adef} and $\sum^*$ denotes summation over square-free elements of $\mz[i]$ congruent to $1 \pmod {16}$.
\end{theorem}

One important difference between our Theorem~\ref{firstmoment} and the results in \cites{FHL, Diac} is that we have the additional restriction to square-free elements of $\intz[i]$ in the average.  Another significant distinction between our result and those in \cite{FHL} is the authors of \cite{FHL} study the mean values of highly modified $L$-series associated fixed order Hecke character.   Moreover, the result in \cite{Diac} ``does not seem to be accessible by more traditional methods" while our result here is obtained via classical means.   We also note here that second moment and non-vanishing results for Hecke $L$-functions associated with characters of an arbitrary fixed order were proved in \cite{BGL}.

\subsection{Strategy of the proof}  To prove \eqref{1stmom}, $L(1/2, \chi_c)$ is first decomposed using the approximate functional equation \eqref{approxfuneq}.  We are then led to study two sums, \eqref{sum1} and \eqref{sum2}.  After various manipulations and transformations, the sum in \eqref{sum1} will give the main term on the right-hand side of \eqref{1stmom}.  The large sieve for quartic Hecke characters, Lemma~\ref{quartls}, plays an important role in controlling the error terms in the approximation.  The sum in \eqref{sum2} relies crucially on an estimate of a smooth Gauss sum, \eqref{bound}.  Theorem \ref{firstmoment} follows from the combination and optimization of the above-mentioned estimates. \newline

One should be able to establish a result similar to Theorem~\ref{firstmoment} if $1/2$ on the left-hand side of \eqref{1stmom} is replaced by $1/2+it$ and $c$ by a set of more general set of numbers satisfying some almost-prime constraints.  We leave the details of these possible extensions the enthused reader.

\section{Preliminaries}
\label{sec 2}

In this section, we collect the information needed in the proof of our main result.

\subsection{Quartic symbol and quartic Gauss sum}
\label{sec2.4}
   The symbol $(\frac{\cdot}{n})_4$ is the quartic
residue symbol in the ring $\mz[i]$.  For a prime $\pi \in \mz[i]$
with $N(\pi) \neq 2$, the quartic character is defined for $a \in
\mz[i]$, $(a, \pi)=1$ by $\leg{a}{\pi}_4 \equiv
a^{(N(\pi)-1)/4} \pmod{\pi}$, with $\leg{a}{\pi}_4 \in \{
\pm 1, \pm i \}$. When $\pi | a$, we define
$\leg{a}{\pi}_4 =0$.  Then the quartic character can be extended
to any composite $n$ with $(N(n), 2)=1$ multiplicatively. \newline

Note that in $\intz[i]$, every ideal coprime to $2$ has a unique generator congruent to 1 modulo $(1+i)^3$.  Such a generator is called primary. Recall that the quartic reciprocity law (Theorem 2 on page 123 of \cite{I&R}) states that for two primary primes  $m, n \in \mz[i]$,
\begin{equation*}
 \leg{m}{n}_4 = \leg{n}{m}_4(-1)^{((N(n)-1)/4)((N(m)-1)/4)}.
\end{equation*}
Observe, by Lemma 6 on page 121 of \cite{I&R}, that a non-unit $n=a+bi$ in $\mz[i]$ is congruent to $1 \pmod{(1+i)^3}$ if and only if $a \equiv 1 \pmod{4}$, $b \equiv 0 \pmod{4}$ or $a \equiv 3 \pmod{4}$, $b \equiv 2 \pmod{4}$. \newline

 For a non-unit $n \in \mz[i]$, $n \equiv 1 \pmod {(1+i)^3}$, the quartic
 Gauss sum $g(n)$ and a corresponding quadratic Gauss sum $g_2(n)$ is defined by
\[    g(n) =\sum_{x \bmod{n}} \leg{x}{n}_4 \widetilde{e}\leg{x}{n} \qquad \mbox{and} \qquad g_2(n)  =\sum_{x \bmod{n}} \leg{x}{n}^2_4 \widetilde{e}\leg{x}{n}, \]
  where $ \widetilde{e}(z) =\exp \left( 2\pi i  (\frac {z}{2i} - \frac {\overline{z}}{2i}) \right)$. \newline
  
We now extend the definition of $g(n)$ and $g_2(n)$ to $n=1$ by setting $g(1)=g_2(1)=1$. This is done so that the functional equation \eqref{1.1} is still valid with $W(\chi_1)=g(1)$. Furthermore, the following well-known relation (see \cite{Diac}) now holds for all $n$:
\begin{align}
\label{2.1}
   |g(n)|& =\begin{cases}
    \sqrt{N(n)} \qquad & \text{if $n$ is square-free}, \\
     0 \qquad & \text{otherwise}.
    \end{cases}
\end{align}
   The above relation holds for $g_2(n)$ as well. \newline

   More generally, for a non-unit $n \in \mz[i]$, $n \equiv 1 \pmod {(1+i)^3}$, we set
\begin{equation*}
 g(r,n) = \sum_{x \bmod{n}} \leg{x}{n}_4 \widetilde{e}\leg{rx}{n}.
\end{equation*}

    The following properties of $g(r,n)$ can be found in \cite{Diac}:
\begin{align}
\label{eq:gmult}
 g(rs,n) & = \overline{\leg{s}{n}}_4 g(r,n), \quad (s,n)=1, \\
\label{2.03}
   g(r,n_1 n_2) &=\leg{n_2}{n_1}_4\leg{n_1}{n_2}_4g(r, n_1) g(r, n_2), \quad (n_1, n_2) = 1, \\
\label{2.04}
g(\pi^k, \pi^l)& =\begin{cases}
    N(\pi)^kg(\pi) \qquad & \text{if} \qquad l= k+1, k \equiv 0 \pmod {4},\\
    N(\pi)^kg_2(\pi) \qquad & \text{if} \qquad l= k+1, k \equiv 1 \pmod {4},\\
    N(\pi)^k\leg{-1}{\pi}_4\overline{g}(\pi) \qquad & \text{if} \qquad l= k+1, k \equiv 2 \pmod {4},\\
    -N(\pi)^k, \qquad & \text{if} \qquad l= k+1, k \equiv 3 \pmod {4},\\
      \varphi(\pi^l)=\#(\mz[i]/(\pi^l))^* \qquad & \text{if} \qquad  k \geq l, l \equiv 0 \pmod {4},\\
      0 \qquad & \text{otherwise}.
\end{cases}
\end{align}

     From the supplement theorem to the quartic reciprocity law (see for example, Lemma 8.2.1 and Theorem 8.2.4 in \cite{BEW}),
we have for $n=a+bi$ being primary,
\begin{align}
\label{2.05}
  \leg {i}{n}_4=i^{(1-a)/2} \qquad \mbox{and} \qquad  \hspace{0.1in} \leg {1+i}{n}_4=i^{(a-b-1-b^2)/4}.
\end{align}
   It follows that for any $c \equiv 1 \pmod {16}$, we have
\begin{align}
\label{2.5}
  \chi_c(1+i)=1.
\end{align}

\subsection{The approximate functional equation}

    Let $x > 1$. By evaluating the integral
\begin{equation*}
   \frac {1}{2\pi i}\int\limits_{(2)}(2\pi)^{-(s+1/2)}\Gamma \left( s+\frac{1}{2} \right)L \left( s+\frac{1}{2}, \chi_c \right) x^s \frac{\dif s}{s}
\end{equation*}
   in two ways, we derive the following expression for $L(1/2, \chi_c)$:
\begin{equation} \label{approxfuneq}
 L \left( \frac{1}{2}, \chi_c \right) = \sum_{0 \neq \mathcal{A} \subset
  \mathcal{O}_K}\frac{\chi_c(\mathcal{A})}{N(\mathcal{A})^{1/2}}\Gamma \left( \frac{1}{2}, \frac{2\pi}{x} N(\mathcal{A}) \right) \\
   + \frac{W(\chi_c)}{N(c)^{1/2}} \sum_{0 \neq \mathcal{A} \subset
  \mathcal{O}_K}\frac{\overline{\chi}_c(\mathcal{A})}{N(\mathcal{A})^{1/2}}\Gamma\left( \frac{1}{2}, \frac{2\pi
  N(\mathcal{A})x}{|D_K|N(c)} \right),
  \end{equation}
   where
\begin{align*}
  \Gamma \left( \frac{1}{2}, \xi \right)=\frac {1}{2\pi
   i}\int\limits_{(2)}\frac {\Gamma(s+1/2)}{\Gamma (1/2)}\frac
   {\xi^{-s}}{s} \ \dif s=\frac
   {1}{\Gamma(1/2)}\int\limits^{\infty}_{\xi}t^{-1/2}e^{-t} \ \dif t
\end{align*}
  is the (normalized) incomplete $\Gamma$-function. \newline

  It is easy to see that
\begin{align}
\label{2.07}
      \Gamma \left( \frac{1}{2}, \xi \right) = & \begin{cases}
     1+O(\xi^{1/4})\qquad & 0<\xi<1, \\
      O(e^{-\xi}) \qquad & \xi \geq 1.
    \end{cases}
\end{align}

\subsection{Average of a smooth Gauss sum}
\label{section: smooth Gauss}

    One crucial ingredient of this paper is a bound for the average of certain smoothed Gauss sums. \newline

    Before stating the result, we first introduce a few notations. For $(c,a)=1$, set
\begin{align*}
   \psi_a(c)=(-1)^{((N(a)-1)/4)((N(c)-1)/4)}.
\end{align*}
   It is easy to see that $\psi_a(c)$ is a ray class character $\pmod {16}$. \newline

    We use the convention in this subsection that all sums over elements of $\mz[i]$ are restricted to elements congruent to $1 \pmod {(1+i)^3}$ and we use $\pi$ to denote a prime in $\mz[i]$. Also let $\mu_{[i]}$ stand for the M\"obius function on $\mz[i]$. For any ray class character $\pmod {16}$, we set
\begin{align*}
   h(r,s;\chi)=\sum_{(n,r)=1}\frac {\chi(n)g(r,n)}{N(n)^s}.
\end{align*}

   For any square-free, non-unit $a \in \mz[i]$, let $\{\pi_1, \cdots, \pi_k \}$ be the set of distinct prime divisors of $a$.  We define
\begin{align*}
  P(a)=\prod^{k}_{i=1}\overline{\leg{(a/\prod^{i-1}_{j=1}\pi_j)^2}{\pi^3_i}}_4,
\end{align*}
   where we set the empty product to be $1$. As $\leg{\pi^2_2}{\pi^3_1}_4=\leg{\pi_2}{\pi_1}^2_4$ for two distinct primes $\pi_1, \pi_2$, one checks easily by induction on the number of prime divisors of $a$ that $P(a)$ is independent of the order of $\{\pi_1, \cdots, \pi_k \}$.

\begin{lemma}
\label{lem2} Suppose $f , \alpha$ are square-free and $(r, f ) = 1$.  Set
\begin{equation*}
  h(r,f,s;\chi)=\sum_{(n,rf)=1}\frac {\chi(n)g(r,n)}{N(n)^s}  \qquad  \mbox{and} \qquad h_{\alpha}(r,s;\chi)=\sum_{(n,\alpha)=1}\frac {\chi(n)g(r,n)}{N(n)^s}.
\end{equation*}
  Furthermore suppose $r =r_1r^2_2r^3_3r^4_4$ where $r_1r_2r_3$ is square-free, and let $r^*_4$ be the product of primes dividing $r_4$. We define
\begin{equation} \label{hstardef}
 h^*_{r_1}(r_1r_2^2r^3_3,s;\chi)= \sum_{a|r_2}\mu_{[i]}(a)\chi(a)^3N(a)^{2-3s}\overline{\leg{-r_1(r_2/a)^2r^3_3}{a^3}}_4 P(a) \left( \prod_{\pi| a}\overline{g}(\pi) \right)    h_{r_1}(r_1r_2^2r^3_3/a^2,s;\psi_{a^3}\chi),
\end{equation}
  where we define $\leg{\cdot}{1}_4=1$ in the sum above and the empty product to be $1$.  Then
\begin{align}
\label{2.11}
  h(r,f,s;\chi)&=\sum_{a | f}\frac {\mu_{[i]}(a)\chi(a)g(r,a)}{N(a)^s}h(a^2r,s;\psi_a\chi), \\
\label{2.12}
  h(r_1r^2_2r^3_3r^4_4,s;\chi) &=h(r_1r^2_2r^3_3, r^*_4,s;\chi), \\
\label{2.13}
  h(r_1r^2_2r^3_3,s;\chi) &=\prod_{\pi|r_3}(1-\chi(\pi)^4N(\pi)^{3-4s})^{-1}h_{r_1r_2}(r_1r_2^2r^3_3,s;\chi), \\
\label{2.14}
  h_{r_1r_2}(r_1r_2^2r^3_3,s;\chi) =& \prod_{\pi|r_2} \left(1-\psi_{\pi^3}(\pi)\chi(\pi)^4N(\pi)^{2-4s}|g(\pi)|^2\overline{\leg{-1}{\pi^3}}_4 \right )^{-1}  h^*_{r_1}(r_1r_2^2r^3_3,s;\chi), \\
\label{2.15}
  h_{r_1}(r_1r_2^2r^3_3,s;\chi) &=\prod_{\pi|r_1} \left( 1+\chi(\pi)^2N(\pi)^{1-2s}g_2(\pi)\overline{\leg{r_1r_2^2r^3_3/\pi}{\pi^2}}_4 \right)^{-1}h_{1}(r_1r_2^2r^3_3,s;\chi).
\end{align}
\end{lemma}
\begin{proof}
     As the proof is similar to that of \cite[Lemma 3.6]{B&Y}, we only give the proof of \eqref{2.14} here.
  To that end, let $ab^2c^3 \in \mz[i]$ and $\pi$ be a prime in $\mz[i]$ such that $abc\pi$ is square-free. Then
\[ h_{ab\pi}(ab^2c^3\pi^2 ,s;\chi) =\sum_{(n,ab\pi)=1}\frac {\chi(n)g(ab^2c^3 \pi^2,n)}{N(n)^s} \\
  =\sum_{(n,ab)=1}\frac {\chi(n)g(ab^2c^3 \pi^2,n)}{N(n)^s}-\sum_{\substack{\pi | n \\ (n,ab)=1 }}\frac {\chi(n)g(ab^2c^3 \pi^2,n)}{N(n)^s}. \]
   Writing in the latter sum $n = \pi^jn'$ with $(n', \pi) = 1$, we have by \eqref{2.03} that
\begin{align*}
   g(ab^2c^3 \pi^2, \pi^jn') = \leg{\pi^j}{n'}_4\leg{n'}{\pi^j}_4g(ab^2c^3 \pi^2, \pi^j)g(ab^2c^3 \pi^2, n').
\end{align*}
    Using \eqref{eq:gmult} and \eqref{2.04} we see that $g(ab^2c^3 \pi^2, \pi^j)=0$ unless $j=3$, in which case we deduce from quartic reciprocity, \eqref{eq:gmult} and \eqref{2.04} that (note that $\leg{-1}{\pi}_4= \leg{-1}{\pi^3}_4$ by \eqref{2.05})
\begin{align*}
   g(ab^2c^3 \pi^2, \pi^3n') = N(\pi)^2\overline{g}(\pi)\overline{\leg{-ab^2c^3}{\pi^3}}_4 \psi_{\pi^3}(n')g(ab^2c^3, n').
\end{align*}
   This implies that
\begin{equation} \label{2.16}
  h_{ab\pi}(ab^2c^3\pi^2 ,s;\chi)=h_{ab}(ab^2c^3\pi^2,s;\chi)-\chi(\pi)^3N(\pi)^{2-3s}\overline{g}(\pi)\overline{\leg{-ab^2c^3}{\pi^3}}_4  h_{ab\pi}(ab^2c^3 ,s;\psi_{\pi^3}\chi).
\end{equation}

   On the other hand, we have
\begin{align*}
  h_{ab\pi}(ab^2c^3 ,s;\psi_{\pi^3}\chi) &=\sum_{(n,ab\pi)=1}\frac {\psi_{\pi^3}(n)\chi(n)g(ab^2c^3,n)}{N(n)^s} \\
  &=\sum_{(n,ab)=1}\frac {\psi_{\pi^3}(n)\chi(n)g(ab^2c^3,n)}{N(n)^s}-\sum_{\substack{\pi | n \\ (n,ab)=1 }}\frac {\psi_{\pi^3}(n)\chi(n)g(ab^2c^3,n)}{N(n)^s}.
\end{align*}
   Writing in the latter sum $n = \pi^jn'$ , where $(n', \pi) = 1$, then we have by \eqref{2.03} that
\begin{align*}
   g(ab^2c^3, \pi^jn') = \leg{\pi^j}{n'}_4\leg{n'}{\pi^j}_4g(ab^2c^3, \pi^j)g(ab^2c^3, n').
\end{align*}
    Using \eqref{eq:gmult} and \eqref{2.04} we see that $g(ab^2c^3, \pi^j)=0$ unless $j=1$, in which case we deduce from quartic reciprocity, \eqref{eq:gmult} and \eqref{2.04} that
\begin{align*}
g(ab^2c^3, \pi n') = g(\pi)\overline{\leg{ab^2c^3}{\pi}}_4\psi_{\pi}(n')g(ab^2c^3\pi^2, n').
\end{align*}
   This implies that (noting that $\psi_{\pi}\psi_{\pi^3}$ is principal)
\begin{align}
\label{2.17}
  h_{ab\pi}(ab^2c^3 ,s;\chi) =h_{ab}(ab^2c^3 ,s;\psi_{\pi^3}\chi)-\psi_{\pi^3}(\pi)\chi(\pi)N(\pi)^{-s}g(\pi)\overline{\leg{ab^2c^3}{\pi}}_4  h_{ab\pi}(ab^2c^3\pi^2 ,s;\chi).
\end{align}

   We then deduce from \eqref{2.16} and \eqref{2.17} that
\begin{align*}
  h_{ab\pi}(ab^2c^3\pi^2,s;\chi) =& \left( 1-\psi_{\pi^3}(\pi)\chi(\pi)^4N(\pi)^{2-4s}|g(\pi)|^2\overline{\leg{-1}{\pi^3}}_4 \right)^{-1}(h_{ab}(ab^2c^3\pi^2,s;\chi) \\
  & \hspace*{3.5cm} -\chi(\pi)^3N(\pi)^{2-3s}\overline{g}(\pi)\overline{\leg{-ab^2c^3}{\pi^3}}_4  h_{ab}(ab^2c^3,s;\psi_{\pi^3}\chi) \\
\end{align*}
  An induction argument on the number of prime divisors of $b$ gives \eqref{2.14}.
\end{proof}

    Now we deduce readily the following lemma concerning the analytic behavior of $h(r,s;\chi)$ on $\Re(s) >1$.
\begin{lemma}
\label{lem1} The function $h(r,s;\chi)$ has meromorphic continuation to the entire complex plane. It is holomorphic in the
region $\sigma=\Re(s) > 1$ except possibly for a pole at $s = 5/4$. For any $\varepsilon>0$, letting $\sigma_1 = 3/2+\varepsilon$, then for $\sigma_1 \geq \sigma \geq \sigma_1-1/2$, $|s-5/4|>1/8$, we have
\begin{equation} \label{hbound}
h(r,s;\chi) \ll N(r)^{\frac 12(\sigma_1-\sigma+\varepsilon)}(1+t^2)^{ \frac {3}{2}(1+\varepsilon)},
\end{equation}
  where $t=\Im(s)$. Moreover, the residue satisfies
\begin{equation} \label{resbound}
 \mathrm{Res}_{s=5/4}h(r,s;\chi) \ll N(r)^{1/8+\varepsilon}.
 \end{equation}
\end{lemma}
\begin{proof}
We can reduce the estimation of $h(r, s;\chi)$ to that of $h_1(r,s;\chi)$. To see this, we write $r=r_1r_2^2r_3^3 r_4^4$ with $r_1$, $r_2$ and $r_3$ square-free. \eqref{2.12} gives that
   \[ h(r,x;\chi) = h(r_1r_2^2r_3^3, r_4^*, s; \chi)  \]
where $r_4^*$ is the product of primes dividing $r_4$.  Then applying \eqref{2.11}, $h(r_1r_2^2r_3^3, r_4^*, s; \chi)$ can be written as a sum involving $h(a^2r, s; \psi_a \chi)$ with $a$ dividing $r_4^*$.   Repeating this process if needed (as in the case $(r_4, r_2)>1$), we will arrive at an expression involving $h(r_1r_2^2r_3^3, s; \chi)$ with square-free $r_1$, $r_2$ and $r_3$.  Now, we can apply \eqref{2.13}, \eqref{2.14}, \eqref{hstardef} and \eqref{2.15} and end up with an expression involving $h_1(r_1r_2^2r_3^3, s; \chi)$. \newline

Now it follows from the proof of the Lemma on p. 200 of \cite{P} that $h_1(r,s;\chi)$ satisfies the properties of the lemma, this gives us the desired result.
\end{proof}

Now we are ready to state and prove the necessary bound.

\begin{lemma} \label{crubound}
 Let $\widetilde{g}(n) =g(n) N(n)^{-1/2}$. For any rational integer $r \geq 0, a \in \mz[i]$, we have for $y \geq x \geq 1$,
\begin{align} \label{bound}
      \sum_{c \equiv 1 \bmod {16}}\widetilde{g}(c)\overline{\chi}_c(a)\Gamma \left( \frac 12, \frac {2\pi
  2^rN(a)x}{|D_K|N(c)} \right) \exp \left( \frac {-N(c)}y \right) \ll_{\varepsilon}
      N(a)^{1/4+\varepsilon}y^{1/2+2\varepsilon}+y^{3/4+\varepsilon}N(a)^{1/8+\varepsilon}.
\end{align}
\end{lemma}
\begin{proof} Let $\chi$ runs over all ray class characters $\pmod {16}$, we use the ray class characters $\pmod {16}$ to detect the congruence condition
   $c \equiv 1 \pmod {16}$ to obtain
\begin{equation} \label{2.7}
\begin{split}
   & \sum_{c \equiv 1 \bmod {16}}\widetilde{g}(c)\overline{\chi}_c(a)\Gamma \left( \frac 12, \frac {2\pi
  2^rN(a)x}{|D_K|N(c)} \right) \exp \left( -\frac {N(c)}y \right)  \\
   =& \frac {1}{\#h_{(16)}}\sum_{\chi \bmod {16}} \sum_{c \equiv 1 \bmod {(1+i)^3}}\widetilde{g}(c)\chi(c)\bar{\chi}_c(a)\Gamma\left( \frac 12, \frac {2\pi
  2^rN(a)x}{|D_K|N(c)} \right)\exp\left( -\frac {N(c)}y \right)  \\
  =&\frac {1}{\#h_{(16)}}\sum_{\chi \bmod {16}} \frac {1}{2\pi
   i}\int\limits\limits_{(2)}\frac {\Gamma(s'+1/2)}{\Gamma (1/2)} \left( \frac {|D_K|}{2\pi
  2^rN(a)x} \right)^{s'}\sum_{c \equiv 1 \bmod {(1+i)^3}}\frac {g(c)\chi(c)\overline{\chi}_c(a)}{N(c)^{1/2-s'}}\exp\left( - \frac {N(c)}y \right) \frac
   {\dif s'}{s'}.
   \end{split}
\end{equation}
    In view of the exponential decay of $e^{-x}$, we see that the inner sum in the last expression of \eqref{2.7} is holomorphic in the
region $\Re(s') >0$. We can shift the contour of integration in the last expression of \eqref{2.7} to $\Re(s')=\varepsilon >0$, which we now fix. By the Mellin inversion, we can recast the inner sum in the last expression of \eqref{2.7} as
\begin{align}
\label{2.8}
   \frac {1}{2\pi
   i}\int\limits\limits_{(2)}\Gamma(s)y^s\sum_{\substack{(c,a)=1 \\ c \equiv 1 \bmod {(1+i)^3}}}\frac {g(c)\chi(c)\overline{\chi}_c(a)}{N(c)^{1/2-s'+s}} \ \dif s=\frac {1}{2\pi
   i}\int\limits\limits_{(2)}\Gamma(s)h \left( a, \frac 12-s'+s;\chi \right) y^s \ \dif s.
\end{align}
   It follows from Lemma \ref{lem1} and Stirling's formula that by shifting the contour of integration to $\Re(s)=1/2+2\varepsilon$, the integration in \eqref{2.8} is
\begin{align} \label{2.10'}
  \ll N(a)^{1/4+\varepsilon}y^{1/2+2\varepsilon}+y^{3/4+\epsilon}N(a)^{1/8+\varepsilon}.
\end{align}
Here, the first term in the above bound comes from estimating the integral over the line $\Re(s) = 1/2+2 \varepsilon$ and using \eqref{hbound}.  The second term is to bound the residue at $s=5/4$ for which we use \eqref{resbound}.  Now, applying \eqref{2.10'} for the inner sum in the last expression of \eqref{2.7} allows us to readily deduce \eqref{bound}.
\end{proof}

\subsection{The large sieve with quartic symbols}  Another important input of this paper is the following large sieve inequality for quartic Hecke characters.   The study of the large sieve inequality for characters of a fixed order has a long history.  We refer the reader to \cites{DRHB, DRHB1, G&Zhao, B&Y, BGL}.

\begin{lemma} \label{quartls}
Let $M,N$ be positive integers, and let $(a_n)_{n\in \mathbb{N}}$ be an arbitrary sequence of complex numbers, where $n$ runs over $\mz[i]$. Then we have
\begin{equation*}
 \sumstar_{\substack{m \in \mz[i] \\N(m) \leq M}} \left| \ \sumstar_{\substack{n \in \mz[i] \\N(n) \leq N}} a_n \leg{n}{m}_4 \right|^2
 \ll_{\varepsilon} (M + N + (MN)^{2/3})(MN)^{\varepsilon} \sum_{N(n) \leq N} |a_n|^2,
\end{equation*}
   for any $\varepsilon > 0$,where the asterisks indicate that $m$ and $n$ run over square-free elements of $\mz[i]$ that are congruent to $1$ modulo $(1+i)^3$ and $(\frac
{\cdot}{m})_4$ is the quartic residue symbol.
\end{lemma}

\begin{proof}
This is a special case of \cite[Theorem 1.3]{BGL} which improves the bound in \cite[Theorem 1.1]{G&Zhao} (the exponent of $MN$ is $3/4$ in \cite{G&Zhao}).
\end{proof}

\section{Proof of Theorem~\ref{firstmoment}}

\subsection{The main term of the first moment}

We have, using \eqref{approxfuneq} (choosing $x=y^l$, where $0 < l < 1$ will be specified later)
\[ \sumstar_{c \equiv 1 \bmod {16}}L \left( \frac{1}{2},
   \chi_c \right) \exp \left( -\frac{N(c)}{y} \right)={\sum}_1+{\sum}_2, \]
   where
\begin{equation} \label{sum1}
   {\sum}_1 =\sumstar_{c \equiv 1 \bmod {16}} \ \sum_{0 \neq \mathcal{A} \subset
  \mathcal{O}_K} \frac{\chi_c(\mathcal{A})}{N(\mathcal{A})^{1/2}}\Gamma \left( \frac{1}{2}, \frac{2\pi}{x}
  N(\mathcal{A} ) \right) \exp\left( -\frac{N(c)}{y} \right),
  \end{equation}
  and
  \begin{equation} \label{sum2}
    {\sum}_2 =\sumstar_{c \equiv 1 \bmod {16}}\frac{W(\chi_c)}{N(c)^{1/2}} \sum_{0 \neq \mathcal{A} \subset
  \mathcal{O}_K}\frac{\overline{\chi}_c(\mathcal{A})}{N(\mathcal{A})^{1/2}}\Gamma \left( \frac{1}{2}, \frac{2\pi
  N(\mathcal{A})x}{|D_K|N(c)} \right)\exp\left( -\frac{N(c)}{y} \right).
  \end{equation}
   We deal with ${\sum}_1$ first. Since any integral non-zero ideal $\mathcal{A}$ in $\mz[i]$ has a unique generator
$(1+i)^ra$, where $r \in \intz$, $r \geq 0$, $a \in \intz[i]$ and $a \equiv 1 \pmod
{(1+i)^3}$, it follows from the quartic reciprocity law and \eqref{2.5} that
$\chi_{c}(\mathcal{A}) = \chi_a(c)$ (recall our convention that $\chi_1$ is the principal character $\pmod 1$). This enables us to recast ${\sum}_1$ as
\begin{equation} \label{firstsum}
\begin{split}
   {\sum}_1 &= \sumstar_{c \equiv 1 \bmod {16}}\sum_{\substack{r \geq 0 \\ a \equiv 1 \bmod
  {(1+i)^3}}} \frac {\chi_a(c)}{2^{r/2}N(a)^{1/2}}\Gamma \left( \frac{1}{2}, \frac{2\pi}{x}
  2^rN(a) \right) \exp \left( -\frac{N(c)}{y} \right) \nonumber \\
  &= \sum_{r,a}\left( \ \sumstar_{c \equiv 1 \bmod {16}} \chi_a(c)\exp \left(- \frac{N(c)}{y} \right) \right) \frac {\Gamma (1/2, 2\pi
  2^rN(a)/x)}{2^{r/2}N(a)^{1/2}}.
  \end{split}
\end{equation}
   For $a$, a fourth power, the inner sum above is
\begin{equation} \label{integral}
\begin{split}
   \sumstar_{c \equiv 1 \bmod {16}} \chi_a(c)\exp \left( - \frac{N(c)}{y} \right) &=\frac {1}{2\pi
   i}\int\limits_{(2)}\Gamma (s)y^s \left( \ \sumstar_{\substack{c \equiv 1 \bmod {16} \\ (a,c)=1}} \frac {1}{N(c)^s} \right) \dif s \\
   &=\frac {1}{\#h_{(16)}}\sum_{\chi \bmod {16}}\frac {1}{2\pi
   i}\int\limits_{(2)}\Gamma (s)y^s\left( \sum_{\substack{\mathcal{A} \neq 0 \\ (\mathcal{A}, a)=1 }} \frac{\chi(\mathcal{A})|\mu_{[i]}(\mathcal{A})|}{N(\mathcal{A})^s} \right) \dif s,
   \end{split}
\end{equation}
   where $\chi$ runs over all ray class characters $\pmod {16}$.
    Note that we have
\begin{equation*}
 \sum_{\substack{\mathcal{A} \neq 0 \\ (\mathcal{A}, a)=1}}
  \chi(\mathcal{A}) |\mu_{[i]}(\mathcal{A})|N(\mathcal{A})^{-s}
   =
\begin{cases}
    \frac{\zeta_K(s)}{\zeta_K(2s)}\prod_{\mathfrak{p} | (16a)}(1+N(\mathfrak{p})^{-s})^{-1} \qquad & \text{if} \qquad \chi=\chi_0, \\ \\
    L(s, \chi)\prod_{\mathfrak{p}}(1-\chi(\mathfrak{p}^2)N(\mathfrak{p})^{-2s})\prod_{\mathfrak{p} | (16a)}(1+\chi(\mathfrak{p})N(\mathfrak{p})^{-s})^{-1} \qquad & \text{otherwise}.
\end{cases}
\end{equation*}
   Here $\zeta_K(s)$ is the Dedekind zeta function of $K$. Note
   that both $\zeta_K(2s)$ and $\prod_{\mathfrak{p}}(1-\chi(\mathfrak{p}^2)N(\mathfrak{p})^{-2s})$ are holomorphic at $\Re(s) >1$ and
   $O_{\varepsilon} (1)$ for $\Re(s) \geq 1+\varepsilon$.
   Moreover, for a fixed non-principal character $\chi$, let
   $\chi'$ be the primitive character that induces $\chi$. Then it is easy to see that $L(s, \chi)=L(s, \chi')$ since $\chi$ is of conductor $16$.
   By a result of E. Landau \cite{Landau} (see also
\cite[Theorem 2]{Sunley}), which states that for an algebraic
number field $K$ of degree $n \geq 2$,
   $\chi$ any non-principal primitive ideal character of $K$ with conductor $\mathfrak{f}$, $k =|N(\mathfrak{f}) \cdot d_K|$ with $d_K$ being the
discriminant of $K$, we have for $X \geq 1$,
\begin{equation*}
   \sum_{N(I) \leq X}\chi(I) \leq k^{1/(n+1)}\log^n(k) \cdot X^{(n-1)/(n+1)},
\end{equation*}
where $I$ runs over integral ideas of $K$. It follows from this
and partial summation that $L(s,\chi) \ll 1$ when $\Re(s) \geq 1/2$ for any non-principal
ray class character $\chi \pmod {16}$.  Thus on moving the line
of integration on the right-hand side of \eqref{integral} to
$\Re(s) = 1/2+\varepsilon$ and noting that on this line
\begin{align*}
  \prod_{\mathfrak{p} | (16a)}(1+N(\mathfrak{p})^{-s})^{-1} \ll N(a)^{\varepsilon} \; \; \; \; \mbox{and} \; \; \; \;
  \hspace{0.1in} \prod_{\mathfrak{p} |
(16a)}(1+\chi(\mathfrak{p})N(\mathfrak{p})^{-s})^{-1} \ll
N(a)^{\varepsilon},
\end{align*}
  we get that the sum on the left-hand side of \eqref{integral} equals asymptotically
$C_ay+O(y^{1/2+\varepsilon}N(a)^{\varepsilon})$, where
\begin{align*}
  C_a=\frac {\text{res}_{s=1}\zeta_K(s)}{\#h_{(16)}\zeta_K(2)}\prod_{\mathfrak{p} |
  (16a)}(1+N(\mathfrak{p})^{-1})^{-1}.
\end{align*}

   It follows from this and \eqref{2.07} that the contribution from fourth-powers $a$ in \eqref{firstsum} is exactly
\begin{equation} \label{Adef}
Ay+O \left( y^{1+\epsilon}x^{-1/4}+y^{1/2+\varepsilon} \right), \; \mbox{where} \; A=\left( 2+\sqrt{2} \right) \frac {\text{res}_{s=1}\zeta_K(s)}{\#h_{(16)}\zeta_K(2)}
\sum_{(\mathcal{A}, 2)=1 }\frac
{1}{N(\mathcal{A})^{2}}\prod_{\mathfrak{p} |
  (16)\mathcal{A}} \left( 1+N(\mathfrak{p})^{-1} \right)^{-1}.
\end{equation}

\subsection{The remainder terms of the first moment}

   For non-fourth power $a$, $\chi_a$ is non-trivial and we have the analogue of the P\'olya-Vinogradov
   inequality from \cite[Lemma 3.1]{G&Zhao}, that for $y \geq 1$,
\begin{equation}
\label{7.3}
    \sum_{c \equiv 1 \bmod{ (1+i)^3}} \chi_{a}(c)\exp \left( -\frac{N(c)}{y} \right) \ll_{\varepsilon} N(a)^{(1+\varepsilon)/2}.
\end{equation}
   Note that we can assume $N(c) \ll y^{1+\varepsilon}$ and $N(a) \ll x^{1+\varepsilon}$ in view of the exponential decay of the
test functions. We obtain that, after using $\mu_{[i]}$ to detect the
condition that $c$ is square-free,
\begin{align*}
 \sumstar_{c \equiv 1 \bmod {16}} \chi_a(c)\exp \left( -\frac{N(c)}{y} \right)
   =& \sum_{c \equiv 1 \bmod {16}}
   \chi_a(c)\exp \left( -\frac{N(c)}{y} \right) \sum_{\substack{d^2|c \\ d \equiv 1 \bmod
   {(1+i)^3}}}\mu_{[i]}(d) \\
   =& \sum_{\substack{N(d) \leq B \\ d \equiv 1 \bmod
   {(1+i)^3}}}\mu_{[i]}(d)\chi_a(d^2)\sum_{c \equiv \overline{d}^2 \bmod {16}}
   \chi_a(c)\exp \left( -\frac{N(d^2c)}{y} \right) \\
   & \hspace*{1cm} +\sum_{\substack{N(d) > B \\ d \equiv 1 \bmod
   {(1+i)^3}}}\mu_{[i]}(d)\chi_a(d^2)\sum_{c \equiv \overline{d}^2 \bmod {16}}
   \chi_a(c)\exp \left( -\frac{N(d^2c)}{y} \right).
\end{align*}
  In the second sum of the last expression above, we further write $c=c'e^2$ with $c'$
  square-free to arrive at
\begin{align*}
 \sumstar_{c \equiv 1 \bmod {16}} & \chi_a(c)\exp \left( -\frac{N(c)}{y} \right) \\
   =& \sum_{\substack{N(d) \leq B \\ d \equiv 1 \bmod
   {(1+i)^3}}}\mu_{[i]}(d)\chi_a(d^2)\sum_{c \equiv \overline{d}^2 \bmod {16}}
   \chi_a(c)\exp\left( -\frac{N(d^2c)}{y} \right) \\
   & \hspace*{1cm}+\sum_{b \equiv 1 \bmod
   {(1+i)^3}}\chi_a(b^2)\left( \sum_{\substack{d|b, \ N(d) > B \\d \equiv 1 \bmod
   {(1+i)^3}}}\mu_{[i]}(d) \right) \sumstar_{c \equiv \overline{b}^2 \bmod {16}}
   \chi_a(c)\exp \left( -\frac{N(b^2c)}{y} \right) = R+S, \; \mbox{say}.
\end{align*}
   Here $B$ satisfies the bounds $y^{1/2}x^{-1/2} < B <y^{1/2}$ (and hence $x \geq y/N(b)^2$) and will be chosen optimally later. Using the ray
   class characters $\chi \pmod {16}$ to detect the congruence condition
   $c \equiv \overline{d}^2 \pmod {16}$, and applying \eqref{7.3},
   we get
\[ \sum_{c \equiv \overline{d}^2 \bmod {16}}
   \chi_a(c)\exp \left( -\frac{N(d^2c)}{y} \right)
   = \frac {1}{\#h_{(16)}}\sum_{\chi \bmod {16}}\chi(d^2)\sum_{c \equiv 1 \bmod {(1+i)^3}}
   \chi(c)\chi_a(c)\exp \left( -\frac{N(d^2c)}{y} \right) \ll N(a)^{(1+\varepsilon)/2}. \]
   It follows from this that the contribution of $R$ to ${\sum}_1$ is at most $xBy^{\varepsilon}$. Here we note that $\chi\chi_a$ is non-trivial $\pmod {16m}$ as $\chi_a$ is non-trivial $\pmod m$.\newline

   To deal with $S$, we extract square divisors of $a$ by writing $a = a_1a^2_2$,
   where $a_1, a_2 \equiv 1 \pmod {(1+i)^3}$ and $a_1$ is square-free to get
\begin{align*}
   & \sum_{\substack{a \equiv 1 \bmod {(1+i)^3} \\ N(a) \leq x^{1+\varepsilon}}} \
    \sumstar_{\substack{c \equiv \overline{b}^2 \bmod {16} \\ N(c) \leq y^{1+\varepsilon}}}
   \chi_a(c)\exp \left( -\frac{N(b^2c)}{y} \right)\frac {\Gamma(1/2, 2\pi
  2^rN(a)/x)}{N(a)^{1/2}} \\
  =& \sum_{\substack{a_2 \equiv 1 \bmod {(1+i)^3} \\ N(a_2) \leq x^{(1+\varepsilon)/2}}} \
    \sumstar_{\substack{a_1 \equiv 1 \bmod {(1+i)^3} \\ N(a_1) \leq x^{(1+\varepsilon)}N(a_2)^{-2}}} \
    \sumstar_{\substack{c \equiv \overline{b}^2 \bmod {16} \\ N(c) \leq y^{1+\varepsilon}}}
   \chi_{a_1}(c)\chi^2_{a_2}(c)\exp \left( -\frac{N(b^2c)}{y} \right)\frac {\Gamma(1/2, 2\pi
  2^rN(a_1a^2_2)/x)}{N(a_1a^2_2)^{1/2}}.
\end{align*}
  We deduce, by Cauchy's inequality, that
\begin{equation} \label{7.5}
\begin{split}
  \sumstar_{\substack{a_1 \equiv 1 \bmod {(1+i)^3} \\ N(a_1) \leq x^{(1+\varepsilon)}N(a_2)^{-2}}} & \frac {\Gamma(1/2, 2\pi
  2^rN(a_1a^2_2)/x)}{N(a_1)^{1/2}}
    \sumstar_{\substack{c \equiv \overline{b}^2 \bmod {16} \\ N(c) \leq y^{1+\varepsilon}}}
   \chi_{a_1}(c)\chi^2_{a_2}(c)\exp \left( -\frac{N(b^2c)}{y} \right) \\
   \leq & \left( \ \sumstar_{\substack{a_1 \equiv 1 \bmod {(1+i)^3} \\ N(a_1) \leq x^{(1+\varepsilon)}N(a_2)^{-2}}}\frac {\Gamma^2(1/2, 2\pi
  2^rN(a_1a^2_2)/x)}{N(a_1)} \right)^{1/2}  \\
  & \hspace*{2cm} \times \left( \ \sumstar_{\substack{a_1 \equiv 1 \bmod {(1+i)^3} \\
N(a_1) \leq x^{(1+\varepsilon)}N(a_2)^{-2}}}
  \left| \ \sumstar_{\substack{c \equiv \overline{b}^2 \bmod {16} \\ N(c) \leq y^{1+\varepsilon}}}
   \chi_{a_1}(c)\chi^2_{a_2}(c)\exp \left( -\frac{N(b^2c)}{y} \right) \right|^2 \right)^{1/2}.
\end{split}
\end{equation}
   For the first factor on the right-hand side of \eqref{7.5}, we
   have
\begin{align*}
  \sumstar_{\substack{a_1 \equiv 1 \bmod {(1+i)^3} \\
N(a_1) \leq x^{(1+\varepsilon)}N(a_2)^{-2}}}\frac {\Gamma^2(1/2, 2\pi
  2^rN(a_1a^2_2)/x)}{N(a_1)} \ll \sumstar_{\substack{a_1 \equiv 1 \bmod {(1+i)^3} \\
N(a_1) \leq x^{(1+\varepsilon)}N(a_2)^{-2}}}\frac {1}{N(a_1)} \ll
\log x.
\end{align*}

   For the second factor on the right-hand side of \eqref{7.5}, we
   have, by Lemma~\ref{quartls},
\begin{align*}
&   \sumstar_{\substack{a_1 \equiv 1 \bmod {(1+i)^3} \\
N(a_1) \leq x^{(1+\varepsilon)}N(a_2)^{-2}}}
  \left| \ \sumstar_{\substack{c \equiv \overline{b}^2 \bmod {16} \\ N(c) \leq y^{1+\varepsilon}}}
   \chi_{a_1}(c)\chi^2_{a_2}(c)\exp \left( -\frac{N(b^2c)}{y} \right) \right|^2 \\
  \ll & \sumstar_{\substack{a_1 \equiv 1 \bmod {(1+i)^3} \\ N(a_1) \leq x^{(1+\varepsilon)}N(a_2)^{-2}}}
  \left| \ \sumstar_{\substack{c \equiv \overline{b}^2 \bmod {16} \\ N(c) \leq (y/N(b^2))^{1+\varepsilon}}}
  \chi_{a_1}(c)\chi^2_{a_2}(c)\exp\left( -\frac{N(b^2c)}{y} \right) \right|^2 \\
   \ll &
   y^{\varepsilon}\left( \frac{x^{1+\varepsilon}}{N(a_2)^{2}}+\left( \frac{y}{N(b^2)} \right)^{1+\varepsilon}+ \left( \frac{x^{1+\varepsilon}}{N(a_2)^2}  \left( \frac{y}{N(b^2)} \right)^{1+\varepsilon} \right)^{2/3} \right)  \sumstar_{\substack{c \equiv \overline{b}^2 \bmod {16}
\\ N(c) \leq (y/N(b^2))^{1+\varepsilon}}}    \left| \chi^2_{a_2}(c)\exp \left( -\frac{N(b^2c)}{y} \right) \right|^2 \\
  \ll & \left( \frac {x}{N(a_2)^2}+\frac {y}{N(b)^2}+ \left( \frac
{xy}{N(a_2b)^2} \right)^{2/3} \right)\frac {y^{1+4\varepsilon}}{N(b)^2},
\end{align*}

  It now follows that the left-hand side expression of \eqref{7.5} is
\[ \ll y^{\varepsilon/2}\left( \left( \frac {x}{N(a_2)^2}+\frac {y}{N(b)^2}+ \left( \frac
{xy}{N(a_2b)^2} \right)^{2/3} \right) \frac {y^{1+4\varepsilon}}{N(b)^2} \right)^{1/2} \ll y^{3\varepsilon}\left(\frac {(xy)^{1/2}}{N(a_2b)}+\frac
{y}{N(b)^2}+\frac {x^{1/3}y^{5/6}}{N(a_2)^{2/3}N(b)^{5/3}} \right). \]
  From this we deduce that the contribution of $S$ to ${\sum}_1$ is at most
\begin{align*}
 &  \sum_{\substack{N(b) \leq y^{1/2+\varepsilon} \\ b \equiv 1 \bmod
   {(1+i)^3} }}\left( \sum_{\substack{d|b, \  N(d) > B \\d \equiv 1 \bmod
   {(1+i)^3}}}1\right) \sum_{\substack{ N(a_2) \leq x^{(1+\varepsilon)/2} \\ a_2
\equiv 1 \bmod {(1+i)^3} }}
\frac {y^{3\varepsilon}}{N(a_2)}\left( \frac
{(xy)^{1/2}}{N(a_2b)}+\frac
{y}{N(b)^2}+\frac {x^{1/3}y^{5/6}}{N(a_2)^{2/3}N(b)^{5/3}} \right) \\
&=\sum_{\substack{N(d) > B \\ d \equiv 1 \bmod
   {(1+i)^3}}}\sum_{\substack{d|b , N(b) \leq y^{1/2+\varepsilon}\\ b \equiv 1 \bmod
   {(1+i)^3}}}\sum_{\substack{N(a_2) \leq x^{(1+\varepsilon)/2} \\ a_2
\equiv 1 \bmod {(1+i)^3} }}
\frac {y^{3\varepsilon}}{N(a_2)}\left(\frac
{(xy)^{1/2}}{N(a_2b)}+\frac {y}{N(b)^2}+\frac {x^{1/3}y^{5/6}}{N(a_2)^{2/3}N(b)^{5/3}} \right) \\
& \ll y^{6\varepsilon}\left( (xy)^{1/2}+\frac {y}{B}+\frac
{x^{1/3}y^{5/6}}{B^{2/3}} \right).
\end{align*}
   Combining the bounds for $R$ and $S$, we obtain (with a different $\varepsilon$)
\begin{align*}
   {\sum}_1 =Ay+O\left( y^{\varepsilon}\Big ( \frac{y}{x^{1/4}}+xB+ (xy)^{1/2}+\frac {y}{B}+\frac
{x^{1/3}y^{5/6}}{B^{2/3}}  \Big ) \right).
\end{align*}
   We now optimize $B$ to satisfy $xB= x^{1/3}y^{5/6}B^{-2/3}$.   Thus $B=y^{1/2}x^{-2/5}$.  Note
further that when $B>y^{1/2}x^{-1/2}$, we have $y/B<(xy)^{1/2}$. Therefore, for the $B$ thus chosen, we have
\begin{align}
\label{sum1est}
   {\sum}_1 =Ay+O\left( y^{\varepsilon}\left( \frac{y}{x^{1/4}} + y^{1/2}x^{3/5} \right) \right).
\end{align}

   Next we need to bound
\begin{align*}
   {\sum}_2 =\sumstar_{c \equiv 1 \bmod {16}} \frac{W(\chi_c)}{N(c)^{1/2}} \sum_{0 \neq \mathcal{A} \subset
  \mathcal{O}_K} \frac{\overline{\chi}_c(\mathcal{A})}{N(\mathcal{A})^{1/2}}\Gamma \left( \frac{1}{2}, \frac{2\pi
  N(\mathcal{A})x}{|D_K|N(c)} \right)\exp \left( -\frac{N(c)}{y} \right).
\end{align*}
Mindful of the equality $W(\chi_c)=g(c)$ and the fact from \eqref{2.1} that
    $g(c)$ is supported on square-free numbers , we may drop the restriction $^*$ in the outer sum
    above.  Recalling the definition of $\widetilde{g}(c)$ in Lemma \ref{crubound}, we recast
    ${\sum}_2$ as
\begin{align*}
   {\sum}_2 =\sum_{\substack{r \geq 0 \\ a \equiv 1 \bmod {(1+i)^3}}} \frac {1}{2^{r/2}N(a)^{1/2}}
   \sum_{c \equiv 1 \bmod {16}}\widetilde{g}(c)\overline{\chi}_c(a)\Gamma \left( \frac{1}{2}, \frac{2\pi
  2^rN(a)x}{|D_K|N(c)} \right) \exp\left( -\frac{N(c)}{y} \right).
\end{align*}

   Again we can assume $2^rN(a) \ll y^{1+\varepsilon}/x$ by virtue of the exponential decay of the
test functions. Applying the bound \eqref{bound}, we see immediately that
\begin{align}
\label{sum2est}
   {\sum}_2 \ll_{\varepsilon} \sum_{\substack{N(a) \leq y^{1+\varepsilon}/x \\ a \equiv 1 \bmod {(1+i)^3} }}\frac {N(a)^{1/4+\varepsilon}y^{1/2+2\varepsilon}+y^{3/4+\varepsilon}N(a)^{1/8+\varepsilon}}{N(a)^{1/2}}
   \ll y^{4\varepsilon} \left( \frac {y^{5/4}}{x^{3/4}}+\frac {y^{11/8}}{x^{5/8}} \right).
\end{align}

   From \eqref{sum1est} and \eqref{sum2est}, we conclude that
\[  \sumstar_{c \equiv 1 \bmod {16}}L \left( \frac{1}{2},
   \chi_c \right)\exp\left( -\frac{N(c)}{y} \right)  =Ay+O \left( y^{\varepsilon}\left( \frac{y}{x^{1/4}} + y^{1/2}x^{3/5}+\frac {y^{5/4}}{x^{3/4}}+\frac {y^{11/8}}{x^{5/8}} \right) \right)  =Ay+O\left( y^{13/14+\varepsilon} \right), \]
   on taking $x=y^{5/7}$. \newline

\noindent{\bf Acknowledgments.} P. G. is supported in part by NSFC grants 11371043 and 11871082 and L. Z. by the FRG grant PS43707.  Parts of this work were done when P. G. visited the University of New South Wales (UNSW) in June 2017. He wishes to thank UNSW for the invitation, financial support and warm hospitality during his pleasant stay.  Finally, we would like to thank the anonymous referee for his/her careful reading of the paper and many helpful comments.

\bibliography{biblio}
\bibliographystyle{amsxport}

\vspace*{.5cm}

\noindent\begin{tabular}{p{8cm}p{8cm}}
School of Mathematics and Systems Science & School of Mathematics and Statistics \\
Beihang University & University of New South Wales \\
Beijing 100191 China & Sydney NSW 2052 Austrlia \\
Email: {\tt penggao@buaa.edu.cn} & Email: {\tt l.zhao@unsw.edu.au} \\
\end{tabular}

\end{document}